\documentclass{amsart}

\usepackage{amssymb,latexsym} %
\theoremstyle{plain}
\newtheorem{theorem}{Theorem}
\newtheorem{lemma}{Lemma}

\theoremstyle{definition}
\newtheorem{definition}{Definition}

\theoremstyle{remark}

\numberwithin{equation}{section}
\begin{document}
\title[Yang Mills Dirac]{Removable Point Singularities for the Yang Mills Dirac Equations in Two Dimensions}
\author{Penny Smith}
\curraddr{
Department of Mathematics Lehigh University 14 E. Packer Ave, Bethehem Pa. 18015 }
\email{
ps02@lehigh.edu }
\date{
September 15, 2017
}
\today

\dedicatory{Dedicated to the memory of Tom Otway, who was our collaborator on this project, and died tragically in 2016}
\begin{abstract}
In \cite{PS90}, we proved a removable point singularity theorem for the coupled Yang-Mills Higgs Equations over a two dimensional base manifold. Here, we prove a similar theorem for the coupled Yang Mills Dirac Equations.
\end{abstract}
\maketitle
\section{Introduction}\label{S:Introduction}
In \cite{PS90}, a removable point singularities theorem was proved for static Yang Mills Higgs fields in two dimensions. A different system of nonlinear partial differential equations occurs if one replaces the Higgs Field, which is a scalar field, with a spinor field. Whereas the Higgs Field represents the field induced by a Boson, the spinor field represents the field induced by a fermion. Elements $\varphi $ of the fermion field are eigenvectors of the Dirac operator $DJ $ , and the mass $m$ of the particle is represented by the corresponding eigenvalue. Geometrically these eigenvectors are sections of twisted spin bundles and their eigenvalues are sections of trivial real line bundles.

In dimension two, the Dirac equation 

\begin{equation}\label{E:Dirac Equation}
DJ \varphi  =m\varphi 
\end{equation}

reduces to the eigenvalue equation for the twisted Cauchy Riemann operator. The non--triviality of the bundle introduces a curvature term into the second order equations

\begin{equation}\label{E:second order equations}
DJ ^{2}\varphi  =\underline{D} \ast \underline{D}\varphi  +\frac{1}{2}e^{i}\bullet e^{j}\bullet F_{ij}\left (\varphi \right ) =m\varphi 
\end{equation} 

   Here, the Lie algebra valued two-form $F$ is a local representation of the curvature of the vector bundle $E$ ; $\varphi $ is a smooth section of $V \otimes E$, where $V$ is a four dimensional vector space isomorphic to the Clifford algebra of $E$ ; $\left [e_{i}\right ]$ is an orthonormal basis on the Euclidean base space $\Omega $ ; $\underline{D}$ is the spinor total covariant derivative.

Detailed discussions of the Dirac equation, and of equation \eqref{E:second order equations}  are given in \cite{TP82}. An outline also occurs in \cite{TO87}. In two dimensions $\varphi $ can be considered a complex-valued section.

Equation \eqref{E:second order equations} is coupled to the system

\begin{equation}\label{E:J equation}
D \ast F =J\left(\varphi \right) 
\end{equation}

, where $D$ is the exterior covariant derivative ; $ \ast  :\Lambda ^{p} \rightarrow \Lambda ^{n -p}$ is the Hodge involution ; the current $J$ is the 1-form given by

\begin{equation}\label{E:definecurrent}
J\left (\varphi \right ) = -\frac{1}{2}\left \langle \varphi  ,e^{i}\bullet e^{j}\bullet \rho \left (\sigma ^{\alpha }\right )\phi \right \rangle \sigma _{\alpha } \otimes e_{i}
\end{equation}

In equation \eqref{E:definecurrent} $\varrho $ is a unitary representation of the gauge group $G$ : the structure group of a principle bundle $P$ related to $E$ by $adE =P \times _{G} \mathcal{G}$ , and $\mathcal{G}$ is the Lie algebra of $G$  ; $\sigma ^{\alpha }$ is an orthonormal basis for the Lie Algebra, which is the fiber of $adE$. For details, see \cite{AJCT80}, and \cite{TP82}.

The problem is to remove a possible point singularity from an otherwise $C^{\infty }$ solution of equations \eqref{E:second order equations}, and \eqref{E:J equation}. Without loss of generality, we can take $\Omega $ to be a disc $D_{r}(0)$ , and take the singularity to lie at the origin. Our main ideas come from \cite{KKU82} , in which the problem is solved for a base space of dimension $n =4$ and $\phi  =0$. Modifications introduced in \cite{AJCT80}-\cite{PS90} and \cite{TO87} are also used. The main technical interest in dimension two lies in the fact that Serrin test function arguments fail, even for the particle field. But this dimension is interesting for other reasons as well. Since, we always assume that $F \in L^{\frac{n}{2}}$, the Holder inequality with $p =\frac{n}{2}$ and $q =\frac{n}{n -2}$ is used in the higher dimensional arguments in several contexts. This  technique also fails in two dimensions. There is no gauge in which we have an a--priori estimate on $\Vert A\Vert _{H^{1 ,2}}$ in terms of a ( finite)  curvature norm. Finally, the Yang Mills Dirac problem, like the corresponding scalar problem, has a topological obstruction to smoothness in the nontrivial holonomy of the punctured disc. Thus, we need a condition on the bundle connection $\Gamma $ ( this idea was introduced in \cite{PS90}). 

Let the loop $\underset{ -}{l} \colon
[0 ,1] \rightarrow S_{R}^{1}$ be    
given explicitly by $\underset{ -}{l_{R}} :t \rightarrow \left(R\cos 2 \pi t , R\sin 2 \pi t \right)$ , where $S_{R}^{1} = \{x \in (\mathbb{R})^{2} \mid \Vert x \Vert = R\}$. Then, for a chosen point $p$ on the fiber over $(R ,0)$ there is a unique $\Gamma $- horizontal lift of $\underset{ -}{l}_{R}$. Faithful right action by elements of $G$ induces parallel transport on this lift. The particular group element that corresponds to transport of the point $p$ around $\underset{ -}{l}_{R}$ will be denoted by $g(R)$. We represent the path $C_{R}$ by the map $C_{R} :(0 ,4) \rightarrow G$ , corresponding to $R \mapsto G\left(R\right)$ and consider the elements $g(R)$ of $G$ as points on the carrier of the path $C_{R}$. If  $lim_{R \rightarrow 0} g(R) = Id$  we say that $\Gamma $ satisfies a Holonomy condition denoted by Condition
$ H$.

As in Theorem 1.1 of \cite{PS90}, we have
\begin{theorem} \label{T:Halternate}
The following is equivalent to Condition H:
There exists a trivialization over a small ball $B_{R_{0}
-\{0 \} }
$, $\exists R_{0}$ , $0 \leq R_{0} \leq 4 $, centered at the origin, in which the connection defines a local covariant derivative 
\begin{equation}
D= d+A 
\end{equation}
\begin{equation}
A=A_{r}(r, \theta) +A_{\theta}(r,\theta) d\theta
\end{equation}
\text{, with} 
\begin{equation}
A_{r}(r, \theta) \, ,\, A_{\theta}(r,\theta)
\in \Gamma( \mathcal{G} \otimes T^{*}(B_{R_{0} -{0} } ) )
\end{equation}
\text{and, with}
\begin{equation}
lim_{r \rightarrow 0} A_{\theta}(r,\theta) =0 
\end{equation}
\text{,with the limit taken in the sup-norm topology on}
G
\end{theorem}

There is an arithmetic error in the estimate of line -9 page 512 of \cite{PS90}.
In the appendix of the current paper, we use an argument of Karen Uhlenbeck to obtain the estimate of line -9 page 512 of \cite{PS90}.

We now state our main result
\begin{theorem}\label{T:main theorem}
Let $ (F, \, \phi) $ be a $C^{\infty}$ solution of the (second order) Yang-Mills-Dirac equations \eqref{E:second order equations}, and \eqref{E:J equation}, in the punctured ball
$B_{0} (1)- {0}$. Suppose $F \in L^{1}(B_{0} (1)) $, $\phi \in H^{1,2}(B_{0} (1)- {0})$, and suppose that the bundle connection satisfies Condition H. 
Then, $(F, \phi) $ is equivalent via a continuous gauge transformation to to a solution which is $C^{\infty} $ in all of the ball $ B_{0} (1)$.
\end{theorem}

That the condition on $\phi$ is optimal can be seen by considering the special case of a flat connection and small particle mass.
As $m$ tends to zero, equations \eqref{E:second order equations} acquires a fundamental solution $ \phi = ln \vert x \vert $, which just fails to be in $ H^{1,2}(B_{0} (1))$. The condition
$F \in L^{n/2}(B_{0} (1))$ is known to be the natural one for the curvature. (See [\cite{KKU82}, \cite{TP82}--\cite{PS90}, \cite{TO87}]. An alternative condition as in \cite{TOLS87} is actually stronger than the standard one in dimensions two and three.
\section{proof of main theorem}\label{S:proof of main theorem}
In this section, we prove \eqref{T:main theorem}.
We denote by C, dimensional constants which may change from line to line.
\begin{theorem}\label{T:LP bound for phi}
The Dirac field satisfies $\phi \in L^{p} $, for any finite $p$
\end{theorem}
\begin{proof}
By hypothesis, $ \phi \in H^{1, 2-\epsilon} (B_{0} (1)- {0} )$ for any positive $\epsilon < 2$.
 Let $ D^{'}- {0}$ be any punctured disk, concentric with (with center the deleted origin), and strictly contained in $ \, B_{0}(1)- {0}$, and define 
 \[
 \tilde{\phi}=
 \begin{cases}
 \phi, &\text{for $x \in 	D^{'}- {0} $ } \\
 0, & \text{ for $ x \in R^{2} - \{D^{'}- {0} \} $ }
 \end{cases}
\]
The Sobolev embedding theorem gives
\[
H^{1, (2-\epsilon)} (D^{'}- {0} ) \hookrightarrow L^{ (4-2\epsilon)/\epsilon} (D^{'}- {0} )
\]
Choosing $\epsilon$ arbitrarily small, we obtain $\phi$ and thus
$ \tilde{\phi}$ into an arbitrarily high $L^{p}$ space, but not into $L^{\infty}$.
\end{proof}

\begin{lemma}\label{L:LPboundoncurvature}
The curvature $F$ is in $L^{1+\delta}(D^{'}- {0} ) $, for some
$\delta > 0$.
\end{lemma}

\begin{proof}
Define $h= \vert F \vert + m^{2}$ to obtain \cite{TO87}
\[
\triangle h +C (\vert F \vert + \vert \phi \vert \underbar{D} \phi
\vert ) h \geq 0
\]
An improvement of Morrey's theorem for dimension two (cf. Theorem
8.1 of \cite{PS90}, \cite{EB80}, and \cite{NT81} ) gives
\[
\vert h(y) \vert \leq C \Vert h \Vert_{L^{1} ( V_{1/4})}
\]
,where $ V_{1/4} = \{ 1/4 \leq \vert y \vert \leq 3/4 \} $.
Scaling by change of variables as [7], we obtain
\[
r^{2} \vert h(x) \vert \leq C \Vert h \Vert_{L^{1}( V_{r/4}) }
\quad 0 <r <1
\]
(See chapter 7 of \cite{DFKU84} for a general discussion of this technique.)
Thus, on $\vert x \vert =r $, we have

\begin{equation}\label{E:pointwise estimate on F}
r^{2} \vert F \vert \leq C \int h 
\end{equation}

Using a delicate geometric argument, which is now standard in gauge theory, (Uhlenbeck's broken Hodge gauges \cite{KKU82}), we find that
\[
\int_{0 < \vert x \vert \leq 1 } \vert x \vert^{2} \vert F(x) \vert ^{2} \, dx \leq \int_{ \vert x \vert =1 } \vert x \vert^{2} \vert F(x) \vert ^{2} \, dx + R^{4} \int_{0 <\vert x\vert \leq 1}
\vert x \vert^{2} \vert \phi \vert ^{4} \, dx
\]
Scaling again by change of variables, gives for $ 0 < R \leq 1 $, the integral estimate

\begin{equation}\label{E:scaled estimate on integrated curvature}
\int_{0 < \vert y \vert \leq R} \vert y \vert^{2} \vert F(y) \vert^{2} \, dy \leq C (R^{3} R^{-2} ) \int_{ \vert y \vert =
R} \vert y \vert^{2} \vert F(y) \vert^{2} \, dS_{y} + R^{4} \int_{0< \vert y \vert \leq R} \vert \phi \vert^{4} \, dy
\end{equation}

In this estimate, we used the fact that $\phi $ has conformal weight $3/2$ under scale transformations \cite{TP82}, \cite{TO87}. Let
\begin{equation}\label{E:definition of f(R)}
f(R) = \int_{0 < \vert y | \leq R} \vert y \vert^{2} \vert F(y) \vert^{2} \, dy
\end{equation}
Then \eqref{E:scaled estimate on integrated curvature} 
is just 
\[
f(R) \leq C(R f^{'}(R) + R^{4}) \leq C(R f^{'}(R) + R^{2})
\]
Integrating, we find that

\begin{equation}\label{E:integralestimateonF}
\int_{0< \vert x \vert \leq r}  \vert x 
\vert^{2} \vert F(x) \vert^{2} \, dx
\leq  C^{'} r^{1/c} 
\end{equation}
 \quad \text{for} $ c \geq 1 $.
\quad
Combining inequalities \eqref{E:pointwise estimate on F} and \eqref{E:integralestimateonF}, we obtain on the annulus $ V_{r} = \{ r/2 \leq \vert x \vert \leq 2r \} $
, the estimate
\[
\begin{split}
r^{2} \vert F(x) \vert_{ \vert x \vert = r} \leq C \int_{V_{r}}
\vert F(x) + m^{2} \, dx \leq \\
C \left([\int_{V_{r}} 
 \vert x \vert^{-2} \,dx ]^{1/2}
 \vert  x \vert^{2} \vert F(x) \vert^{2} \, dx ]^{1/2} + m^{2} r^{2} \right] \leq \\
C\left( [ \int_{0}^{\pi} \int_{r/2}^{r} \vert x \vert^{-2} \vert x \vert \, d\vert x \vert d\theta]^{1/2} [\int_{V_{r}} \vert x \vert^{2} \vert F(x) \vert^{2} \, dx ]^{1/2} +m^{2} r^{2} ] \right) \leq \\
C r^{1/c}
\end{split}
\]

 where, $0 < r \leq 1$, \quad $ c \geq 1 $.
 Thus, we have for $ 0< \delta^{'} \leq 1 $,
 \[
 r^{2-\delta^{'} }\vert F(x) \vert \leq C 
 \]
 , on the circle $\vert x \vert \leq r $
 
 Integrating in polar coordinates gives the result for any 
 $ \delta < \frac{\delta{'}}{2-\delta^{'}}$.
 \end{proof}
 \begin{lemma}\label{L:subellipticlemma}
 Let $w \in C^{\infty}(B(1)- \{0\} ) \cap H^{1,2}(B(1))$ satisfy the
 sub--elliptic inequality
 \begin{equation}\label{E:subellipticinequality}
 \triangle w + f w \geq 0 \quad , w \geq 0
  \end{equation}
  for $ f \in L^{ 1+ \delta}$, where $ \delta \geq 0 $, and $ w \in L^{p}(B(1)$,
  for all $p < \infty$,
  Then,
  $w$ satisfies
  \begin{equation}\label{E:weaksubellipticinequality}
  \int_{B(1)} (\nabla w  \nabla \eta -f w \eta ) \, dx \leq 0
   \end{equation}
   , for any non--negative $\eta \in C^{\infty}_{0} ( B(1))$. In particular,
   $\eta$ need not vanish in a neighborhood of the origin.
   \end{lemma}
   \begin{proof}
   Choose a sequence $\epsilon_{k} $ of positive numbers, for which $\epsilon_{k} \rightarrow 0$, as $ k \rightarrow \infty $.
   Consider the double test function $ \zeta ( (\bar \eta_{k}) )$, where $\eta
   \in C^{\infty}_{0} (B(1))$ is an arbitrary non--negative function, and
   \begin{equation}
   \bar \eta_{k} =
   \begin{cases}\label{E:definitionnetabarsubk}
   0, & \text{ for} \quad  \vert x \vert \leq \epsilon_{k} \\
   1 & \text{for} \quad  \vert x \vert \geq 1 \\
   \left[ \frac{1}{log ( \frac{1}{\epsilon_{k}})}
   \right] 
   \left[ \frac{1}{log (\frac{\vert x \vert }{\epsilon_{k}})}        \right]
   & \text{for} \quad \epsilon_{k} \leq \vert x \vert \leq 1
   \end{cases}
   \end{equation}
   , then $\zeta$ vanishes in a neighborhood of the origin, and \eqref{E:weaksubellipticinequality} implies that 
   \begin{equation}\label{E:subgradientineq}
   \int_{B(1))} \nabla w \nabla \zeta \, dx \leq \int_{B(1))} \vert f \vert w \zeta \, dx
   	\end{equation}
   	But,
   	\begin{equation}
   	\nabla \zeta = \eta (\nabla \bar{\eta}_{k} )	 + (\nabla \eta) \bar{\eta}_{k}
   	\end{equation}
   	And, by construction $ \bar{\eta}_{k} $ satisfies
   	\begin{equation}\label{E:etalimitswithk}
   	lim_{k \rightarrow \infty}  \bar{\eta}_{k}= 1 \text{ \quad a.e.}
   	\end{equation}
   	and,
   	\begin{equation}\label{E:nablaetalimitwithk}
   	lim_{k \rightarrow \infty} \nabla \bar{\eta}_{k} = 0. \text{ \quad a.e.}
   		\end{equation}
 We have
 \begin{equation}
 \begin{split}
 \int \nabla w \nabla \zeta 	\, dx &\leq \int (\nabla w) (\eta (\nabla \bar{\eta}_{k}) \, dx 
 + \int (\nabla w) (\nabla \eta) \bar{\eta}_{k}\\
 & \leq \int \vert f \vert w \eta \bar{\eta}_{k} \, dx
 \end{split}
 \end{equation}
 or,
 \begin{equation}
 \int (\nabla w ) (\nabla \eta) 	\bar{\eta}_{k} \, dx  \leq C \int \vert f \vert \eta \bar{\eta}_{k} \, dx + \int (\nabla w) \vert \eta \vert \vert \nabla 
 \bar{\eta}_{k} \vert \, dx
 \end{equation}
 
By Holder's inequality 
\begin{equation}
\int \vert f \vert  w \eta \bar{\eta}_{k} \, dx \leq
\left[\Vert f \eta \Vert_{ L^{1+ \delta}}\right]
\left[
\Vert w \bar{\eta}_{k} \Vert_{L^{\frac{1 +\delta}{\delta}   }}
\right] < \infty
\end{equation}
Also, for sufficiently large $k$,

\begin{equation}
\int \vert \nabla w 
\vert \eta \vert \nabla \bar{\eta}_{k} \vert \, dx \leq
C \Vert w \Vert_{H^{1,2}} \Vert \bar{\eta}_{k} \Vert_{H^{1,2}} =0
\end{equation}
Since,
\begin{equation}
\int \vert \nabla \bar{\eta}_{k} \vert^{2} \, dx \rightarrow 0 \quad \text{, as}
\quad k \rightarrow \infty
\end{equation}

This completes the proof of Lemma \eqref{L:subellipticlemma}.

Taking $ w = \vert \phi \vert $ , $ f = \vert F \vert + m^{2}$ ,
we find that \eqref{E:weaksubellipticinequality} is satisfied as a consequence of 
\eqref{E:second order equations}. See  \cite{TO87}, for details.
Using Theorem \eqref{T:LP bound for phi}, Lemma \eqref{L:LPboundoncurvature},
the hypothesis that $\phi \in H^{1,2}(B(1)- \{0 \} )$, and the gradient estimate of Lemma \eqref{L:subellipticlemma}, we conclude that $\phi $ satisfies the hypothesis of Morrey's Theorem (Theorem 5.3.1 of \cite{CM66}).Thus, $\phi$ is bounded on compact subdomains of $B(1)$.

Since, $ F \in L^{p}(B(1))$, for $ p > n/2$, we can make a continuous gauge transformation in $B(1)-\{0 \}$ to a gauge \cite{KKU82}, in which, 
\begin{align}\label{E:goodgaugeeq}
d * A =0 \\
\Vert A \Vert_{H^{1,p}} \leq C \Vert F \Vert_{L^{p}}, \quad p > n/2
\end{align}
In this gauge, we write the field equations \eqref{E:second order equations} and
\eqref{E:definecurrent} in terms of $A$, and obtain an elliptic system
\begin{align}\label{E:ellipticsystemforA}
d *A =0 \\
(\delta d + d \delta) A = J(\phi) - (1/2) \delta [A,A]- *[A,*F]
\end{align}
Using the relation $ F= dA + (1/2) [A,A]$, where $[\quad , \quad ]$, denotes the Lie Bracket, it is easy to show that $A$ is bounded, by using the Morrey's theorem that we previously applied to $ \phi $, with improvements given in
\cite{PS90} to weaken the condition $A \in H^{1,2}(B(1))$. (See also \cite{BGJS81}.)
Arguments exactly analogous to those at the end of Section 3 of \cite{TO87} complete the
proof.
 \end{proof}
 
 \appendix
 \section{Fix of estimate of line -9 page 512 of \cite{PS90} } 
  There is an arithmetic error in line -7 of page 512 of \cite{PS90}.The $1+\epsilon$ exponent on the right hand side there should be $1- \epsilon$.
 This falls just short of the power needed to establish the holonomy decay estimate of line -9 page 512 of \cite{PS90}.
  In order to
 obtain the estimate of line -9 page 512 of \cite{PS90}, we use instead the following argument of Karen Uhlenbeck \cite{KKU2017}.
 Recall Karcher's holonomy estimate of line 10 page 568 of \cite{PS90}.
 \begin{definition}
 Let $F$ be the Lie Algebra valued curvature 2--form of the connection $A$.
 Let $F_{x,y}dx dy$ be the representation of $F$ in orthogonal "rectangular" co-ordinates on the base.
 Let $F_{r, \theta} rdr d \theta$ be the representation of $F$ in polar co-ordinates on the base.
 We define
 \begin{equation}
 f(r) \colon = \int_{0}^{2\pi} \vert 	F_{r, \theta} \vert  \, d \theta
 \end{equation}
 \end{definition}
 
 We will need the following well known fact.
 
 Let $f \colon [a,b] \rightarrow R $. Then, $f \in W^{1,1}[a,b] $ iff $f$ is absolutely continuous.
 To prove this, we first need ( We denote weak derivatives by using $\frac{D}{Dx}$ instead of $\frac{d}{dx}$, replacing $x$ by ${t}$ as needed.)
 \begin{lemma}\label{L:weakderivzeroimpliesconstant}
 Let $f \in W^{1,1}[a,b]$, 
 and, let the weak derivative $ \frac{Df}{Dx} =0$. Then, $f$ is constant.
 (i.e. f has a constant representative in $ L^{1}[a,b]$.)
 \end{lemma}	
 \begin{proof}
For any $\epsilon > 0$, let $ f_{\epsilon}$ be the mollifier of $f$(with compact support compactly contained in $[a,b]$), with mollification parameter $\epsilon$. Then,
\begin{equation}
\frac{df_{\epsilon}}{dx} = \frac{Df_{\epsilon}}{Dx}= 0
\end{equation}
Thus, for each $\epsilon >0 $, we have that $f_{\epsilon}$ is constant on $[a,b]$.
Moreover, $f_{\epsilon} \overset{L^{1}}{\rightarrow} f$. 
But, the space of constant functions on $[a,b]$ is a one dimensional subspace of $L^{1}[a,b]$, hence, it is closed with respect to $L^{1}$ convergence.
Thus, f is constant on $[a,b]$.
\end{proof}
\begin{lemma}\label{L:equivalence}
 Let $f \colon [a,b] \rightarrow \mathbb{R} $. Then, $f \in W^{1,1}[a,b] $ iff $f$ is absolutely continuous.
\end{lemma}
\begin{proof}
$\Leftarrow$
Let $\phi \in C^{\infty}_{0}[a,b]$.
 $f$ absolutely continuous in $[a,b]$ 
 $\Leftrightarrow$ $\phi f$ 
 absolutely continuous in $[a,b]$, hence, the classical derivatives $\frac{df}{dx}$, and $\frac{d \phi f}{dx}$ exist a.e. in $[a,b]$. Thus,
 \begin{align}
 	0=\phi(b) f(b) - \phi(a) f(a) &= \\
 	\int_{a}^{b} \frac{d \phi f}{dx} \, dx 
 	&= \int_{a}^{b} \phi \frac{df}{dx} \, dx+ \int_{a}^{b} f \frac{d \phi}{dx} \, dx
 \end{align}
 , thus
 \begin{equation}
 \int_{a}^{b} f \frac{d \phi}{dx} \, dx = -
 \int_{a}^{b} \phi \frac{df}{dx} \, dx,\quad \forall \phi \in C^{\infty}_{0}[a,b]
\end{equation}
, which implies
 \begin{equation}
	\frac{df}{dx}= \frac{Df}{Dx}\text{ a.e. in $[a,b]$}
	\end{equation}
$\Rightarrow$ Let $f \in W^{1,1}[a,b]$.
Then,
\begin{equation}
	w(x) \colon = \int_{a}^{x} \frac{Df}{Dt} \, dt
	\end{equation}	
	Note $w$ is absolutely continuous in $[a,b]$, and, by the $\Leftarrow$ part of the current lemma, just proved, we have that $w$ has a classical derivative  $\frac{dw}{dx}$ which coincides with its weak derivative
	$\frac{DW}{Dx}$ a.e. in $[a,b]$.
Thus,
\begin{equation}
\frac{Dw}{Dx} =\frac{dw}{dx} = \frac{du}{dx} 
	\end{equation}
	, and thus,
	\begin{equation}
	\frac{D (u - w) }{dx}=0
		\end{equation}
	This implies, by Lemma \ref{L:weakderivzeroimpliesconstant}, that $u-w $ is constant on $[a,b]$. 
	Since, $w$ is absolutely continuous, so is $u$.
	In fact, when $x=a$, we have $u(a) = c+ w(a) =c $, which implies
	$u(x) = u(a) + \int_{a}^{b} \frac{Du}{Dt} \, dt  $
	\end{proof}

 \begin{lemma}\label{UhlenbeckCalculusArgument}
 If $F \in W^{1 , p}(B_{R_{0}}) $, for any $1 < p < 1.5$ then $ esslim_{r \rightarrow 0} \quad (\alpha^{'}(r)) =0$ \quad
 \end{lemma}
 \begin{proof}
 Let $0< r_{0} < R_{0} < 1 $. \\
 Since, $F \in W^{1 , p}(B_{R_{0}})$, we have, using the integral version of Minkowski's inequality, and Holder's inequality, that
 \begin{equation}\label{E:onedimsobolev}
 \int_{0}^{r_{0}} \vert f(\rho) \vert^{p} \rho \, d \rho + \int_{0}^{r_{0}} \vert f^{'}(\rho) \vert^{p} \rho \, d \rho \leq M^{2} < \infty
 	\end{equation}

 For any $r$, with $0 < r \leq r_{0} $, we have that, using Holder's inequality
 \begin{equation}\label{E:basicinequalityforfprime}
  \begin{split}
 \int_{r}^{r_{0} } \vert f^{'} (\rho) \vert \, d \rho  =
\int_{r}^{r_{0} } \left[\vert f^{'} (\rho) \vert  \rho^{p} \right]
\left[ \rho^{-p} \right] \, d \rho  \leq \\
\left[ \int_{r}^{r_{0} } \vert f^{'}  (\rho) \vert^{p}  \rho^{1} \, d \rho \right]^{1/p}
\left[ \int_{r}^{r_{0}} \rho^{ \frac{1}{1- p} } \, d \rho \right]^{\frac
{p-1}{p}} \leq \\
   M^{1/p} \left[ 
(r_{0}) ^{ \frac {2-p}{1-p}}  
- (r)^{\frac{2-p}{1-p} }
\right]
\end{split}
 \end{equation}	
 
 Note, that $f$ is absolutely continuous, since $f \in W^{1,1}(0,1) $, by \ref{E:onedimsobolev}. Thus, 
 \begin{equation}
 f({r_{0}) = f(r) + \int_{r}^{r_{0}} f^{'} (\rho}) \, d \rho
\end{equation}
\begin{equation}\label{E:firstintegralestimateforf}
	f(r) \leq f(r_{0}) +   \int_{r}^{r_{0}} \vert f^{'} (\rho) \vert \, d \rho
\end{equation}
Combining \ref{E:firstintegralestimateforf}, and \ref{E:basicinequalityforfprime}, we obtain
\begin{equation}\label{E:secondintegralestimateforf}
f(r) \leq f(r_{0}) + 
 M^{1/p} \left[ 
(r_{0}) ^{ \frac {2-p}{1-p}}  
- (r)^{\frac{2-p}{1-p} }
\right]
\end{equation}
, and thus
\begin{equation}\label{E:limitofrf(r)}
lim_{r \rightarrow 0} r f(r) \leq  lim_{r \rightarrow 0} r f(r_{0}) + lim_{r \rightarrow 0}
r M^{1/p} \left[ 
(r_{0}) ^{ \frac {2-p}{1-p}} \right] 
- lim_{r \rightarrow 0 }(r)^{\frac{3-2p}{1-p} }
=0
\end{equation}
, where we have used $1 < p < \frac{3}{2}$

Now, recall from Karcher's holonomy formula (line 10 page 568 of \cite{PS90}, using the notation there ) that
\begin{equation}
g(r_{2}) - g(r) = \int_{r}^{r_{2}} \int_{0}^{2 \pi} F_{\rho, \theta} \,  \rho  \, d \theta d \rho
\end{equation}
,and thus
\begin{equation}\label{E:basicgestimate}
\vert g(r_{2})\vert  - \vert g(r) \vert  \leq
\vert \vert g(r_{2})\vert  - \vert g(r) \vert \vert \leq \vert g(r_{2}) - g(r) \vert \leq
\int_{r}^{r_{2}} \int_{0}^{2 \pi} \vert F_{\rho, \theta} \vert \,  \rho \, d \theta d \rho
\end{equation}
Dividing both sides of \ref{E:basicgestimate} by $r_{2}-r $, we obtain
\begin{equation}\label{E:fracgestimate}
\frac{\vert g(r_{2})\vert  - \vert g(r) \vert}{r_{2} -r}  
\leq 
\frac{1}{r_{2} -r} 
\int_{r}^{r_{2}}  f(\rho) \rho \,  d \rho
\end{equation}

Now, letting $ r_{2} \rightarrow r$, in \ref{E:fracgestimate}, and using the absolute continuity of $f$, we obtain
\begin{equation}\label{E:holonomyderivestimate}
g^{'}(r) \leq r f(r)
\end{equation}
Using \ref{E:holonomyderivestimate}, combined with \ref{E:limitofrf(r)}, we finally obtain
\begin{equation}\label{E:holonomyderivativelimit}
\lim_{r \rightarrow 0} g^{'}(r) = 0
\end{equation}
and this proves the required estimate of line -9 page 512 of \cite{PS90}
\end{proof}

\end{document}